\documentclass[final,reqno,a4paper]{amsart}

\usepackage{amsmath,amssymb,amsthm}
\usepackage[foot]{amsaddr}
\usepackage{graphicx}
\usepackage{cite}
\usepackage{enumerate}
\usepackage{nicefrac}
\usepackage{bm}
\usepackage{calrsfs}
\usepackage{a4wide}

\title[Norm-Preserving cG Time Stepping]{A Note On A Norm-Preserving\\Continuous Galerkin Time Stepping Scheme}
\author{Thomas P.~Wihler}
\address{Mathematisches Institut, Universit\"at Bern, Sidlerstrasse 5, CH-3012 Bern, Switzerland}
\email{wihler@math.unibe.ch}
\newcommand{\cG}{\text{\tiny\sf{cG}}}
\newcommand{\NX}[1]{\|#1\|_X}
\newcommand{\dd}{\mathsf{d}}
\newcommand{\A}{\mathcal{A}}
\newcommand{\ip}[1]{\left<#1\right>_X}
\newcommand{\uu}{\bm}
\renewcommand{\P}{\mathbb{P}}
\newcommand{\Pm}{\P^{r_m-1}(I_m;X)}
\newcommand{\Pii}{\P^{r_i-1}(I_i;X)}
\newcommand{\Pp}{\P^{r_m}(I_m;X)}

\newcommand{\CG}{\mathbb{V}^{\bm r}_{\cG}(\M)}

\newcommand{\cU}{U_{\cG}}
\newcommand{\cUp}{\dot{U}_{\cG}}
\newcommand{\dt}{\dd t}
\newcommand{\PG}{\Pi^{\bm r-1}}

\newcommand{\PGp}{\Pi^{r_m}}
\newcommand{\PGm}{\Pi^{r_m-1}}

\newcommand{\M}{\mathsf{M}}
\newcommand{\norm}[1]{\|#1\|_X}
\newcommand{\Norm}[1]{\|#1\|}
\renewcommand{\u}{u}
\newcommand{\up}{\dot{\u}}

\newcommand{\U}{\cU}
\newcommand{\Up}{\dot{U}_{\cG}}
\newcommand{\Uh}{\widehat{U}_{\cG}}
\newcommand{\Uhp}{\dot{\widehat{U}}_{\cG}}
\renewcommand{\S}{\mathsf{S}_{\|\u_0\|_X}}

\newcommand{\eh}{\widehat e_{\cG}}

\renewcommand{\L}{\mathcal L}
\newcommand{\F}{\mathcal F}

\newtheorem{Remark}{Remark}
\newenvironment{remark}{\begin{Remark}\rm}{\end{Remark}}
\newtheorem{theorem}{Theorem}

\newtheorem{proposition}{Proposition}


\begin{document}

\begin{abstract}
In this note we shall devise a variable-order continuous Galerkin time stepping method which is especially geared towards norm-preserving dynamical systems. In addition, we will provide an \emph{a posteriori} estimate for the $L^\infty$-error.
\end{abstract}

\maketitle

\section{Introduction}

Given a (final) time~$T>0$, and a real Hilbert space~$X$, with inner product~$\ip{\cdot,\cdot}$ and induced norm~$\NX{\cdot}$, the focus of this contribution is on continuous Galerkin approximations of dynamical systems,
\begin{align}
  \up(t) &= \F(t,u(t)), \quad t\in(0,T],\qquad \u(0)=\u_0,\label{eq:12}
\end{align}
where $\u:\,[0,T]\to X$ is an unknown solution, and~$u_0\in X$ is a prescribed initial value. We suppose that $\F:\,(0,T]\times X\to X$ is a possibly nonlinear operator which satisfies the \emph{orthogonality} property
\begin{equation}\label{eq:F1}\tag{F1}
\ip{\F(t,v),v}=0\qquad\forall v\in X,\,\forall t\in(0,T].
\end{equation}
We remark that~\eqref{eq:F1} includes the case of \emph{skew-symmetric linear} operators, i.e., $\F(t,v)=\A(t)v$, with~$\A(t)^\star=-\A(t)$, $v\in X$, $t\in(0,T]$. Henceforth, whenever clear from the context, we shall write~$\F(u)$ instead of~$\F(t,u(t))$, thereby suppressing the explicit dependence of~$\F$ on the time variable~$t$.

An important consequence of~\eqref{eq:F1} is the fact that, formally, solutions~$\u$ of~\eqref{eq:12} satisfy
\[
\frac12\frac{\dd}{\dt}\|\u(t)\|^2=\ip{\up(t),\u(t)}=\ip{\F(t,u(t)),\u(t)}=0,
\]
for any~$t\in(0,T]$. Especially, any solution~$\{u(t):\,0\le t\le T\}$ of~\eqref{eq:12} describes a trajectory on the sphere 
\begin{equation}\label{eq:S}
\S=\{v\in X:\, \NX{v}=\NX{\u_0}\},
\end{equation}
i.e., there holds
\begin{equation}\label{eq:0}
\NX{\u(t)}=\NX{\u_0}\qquad \forall t\in[0,T].
\end{equation}
Dynamical systems of this kind appear, for example, in quantum mechanics or optics in the context of the Bloch sphere; see, e.g., \cite{NielsenChuang:00}. 

The aim of this paper is to present and analyze a new variable-order  continuous Galerkin (cG) time stepping discretization scheme for the numerical approximation of the system~\eqref{eq:12} which \emph{preserves} the property~\eqref{eq:0} at the discrete time nodes (see Section~\ref{sc:cG}). In this sense, the proposed method can be seen as a \emph{geometric} integration scheme; see, e.g, ~\cite{HairerLubichWanner:06} for more details on this matter. Moreover, we will prove an \emph{a posteriori} bound for the temporal $L^\infty$-error (Section~\ref{sc:apost}). 

Galerkin time stepping methods for initial-value problems were introduced, for example, in~\cite{Hulme:72,Hulme:72a} (see also~\cite{EstepFrench:94}). These schemes can be seen as implicit one-step methods that are based on weak formulations. They generate piecewise polynomial approximations (of arbitrary degree) in time, and can, thus, be naturally cast into the framework of so-called $hp$-version finite element methods. This methodology, in turn, enables maximal flexibility in the choice of the local approximation degrees and time steps; see, e.g.,~\cite{Wihler:05,HolmWihler:15} for $hp$-version continuous Galerkin time stepping methods for ordinary differential equations, and~\cite{SchotzauWihler:10} for $hp$-cG time approximation schemes for linear parabolic partial differential equations. The remarkable advantage of the $hp$-approach (as compared to, for instance, low-order methods) is its ability to adapt to the local behavior of solutions in a very effective manner, and, thereby, to exhibit \emph{high algebraic} or even \emph{exponential rates of convergence}; see, e.g., \cite{BrunnerSchotzau:06,SchotzauSchwab:00,Wihler:05}.

Throughout this note Bochner spaces will be used: For an open interval $I=(a,b)\subset\mathbb{R}$, $a<b$, and a real Hilbert space~$X$ as before, the space $C^0(\overline{I};X)$ consists of all functions $u:\overline I\to X$ that are continuous on~$\overline{I}$ with values in $X$. Moreover, introducing, for~$1\le p\le\infty$, the norm
\[
\|u\|_{L^p(I;X)}=\begin{cases}
\displaystyle\left(\int_I\|u(t)\|^p_X\dd t\right)^{\nicefrac{1}{p}},&1\le p<\infty,\\[2ex]
\text{ess sup}_{t\in I}\|u(t)\|_X,&p=\infty,
\end{cases}
\]
we write $L^p(I;X)$ to signify the space of measurable functions $u:I\to X$ so that the corresponding norm is bounded. 


\section{$hp$-Continuous Galerkin Time Stepping}\label{sc:cG}

\subsection{Discrete Spaces}
Galerkin time discretization methods are based on a temporal partition $\mathsf M=\{I_m\}_{m=1}^M$ of the time interval $(0,T)$ into $M$ open subintervals $I_m=(t_{m-1}, t_m)$, $m=1,2,\ldots, M$, which are obtained from a set of time nodes~$0=t_0<t_1<t_2<\ldots <t_{M-1}<t_M=T$. We set $k_m=t_m-t_{m-1}$, and refer to $I_m$ as the $m^{\rm th}$ time step. To each time step $I_m$ we assign a polynomial degree $r_m\geq 1$ (taking the role of an approximation order), and store these numbers in a vector $\uu{r}=\{r_m\}_{m=1}^M$. In the sequel, for an integer $\ell$, we write $\uu{r}\pm\ell$ to denote the degree vector $\{r_m\pm\ell\}_{m=1}^M$. Additionally, we define by
\[
\P^r(I;X)=\bigg\{\,p:\,\overline I\to X\,\bigg|\, p(t)=\sum_{i=0}^r
x_i t^i,\ x_i\in X\,\bigg.\bigg\}
\]
the space of all polynomials of degree at most $r\in\mathbb{N}_0$ on an interval $\overline I\subset\mathbb{R}$, with coefficients in $X$. Moreover, for a given time partition~$\M$ and an associated degree vector~$\bm r$, we let
\[
\CG=\left\{u\in C^0([0,T];X)\,\bigg|\, u|_{I_m}\in\P^{r_m}(I_m;X)\,\bigg.\right\}
\]
be the global cG space on~$(0,T)$.

\subsection{A Norm-Preserving cG Scheme}

We define the \emph{continuous Galerkin  time stepping method} for the approximation of~\eqref{eq:12} iteratively  as follows: For a prescribed polynomial degree~$r_m\ge 1$, and a given initial value
\begin{equation}\label{eq:Um1}
\cU^{m-1}:=\cU(t_{m-1})\in X, 
\end{equation}
with~$\cU^0:=u_0$, where~$u_0\in X$ is the initial value from~\eqref{eq:12}, we find~$\cU|_{I_m}\in\mathbb{P}^{r_m}(I_m;X)$ through the weak formulation
\begin{equation}
\label{eq:cgfem} 
\int_{I_m}\ip{\cUp-\F(\PGm\cU),V}\dd t=0
\qquad \forall V\in\P^{r_m-1}(I_m;X),
\end{equation}
for any~$1\le m\le M$. 

We emphasize that, in contrast to the standard cG time marching scheme (see, e.g., \cite{Wihler:05}), the method~\eqref{eq:cgfem} contains a local $L^2$-projection~$\PGm:\,L^2(I_m;X)\to\Pm$ within the application of the operator~$\F$; it is defined by
\begin{equation}\label{eq:L2P}
W\mapsto\PGm W:\quad\int_{I_m}\ip{W-\PGm W,V}\dt=0\quad\forall V\in\Pm.
\end{equation}
Indeed, this ensures the norm-preserving property~\eqref{eq:0} at the time nodes (see Proposition~\ref{pr:0} below). Furthermore, notice that, in order to enforce the initial condition~\eqref{eq:Um1} on each individual time step, and, hence, the \emph{global continuity} of~$\cU$ on~$[0,T]$, the local trial space has one degree of freedom more than the local test space. 

Using the $L^2$-projection from~\eqref{eq:L2P}, we observe that~\eqref{eq:cgfem} can be written as
\[
\int_{I_m}\ip{\cUp-\PGm\F(\PGm\cU),V}\dd t=0
\qquad \forall V\in\P^{r_m-1}(I_m;X).
\]
Hence, since the cG solution~$\U$ from~\eqref{eq:cgfem} is globally continuous on~$[0,T]$, and because~$\Up|_{I_m}\in\P^{r_m-1}(I_m;X)$, $1\le m\le M$, the following strong formulation is satisfied:
\begin{equation}\label{eq:strong}
\begin{split}
\cUp&=\PG\F(\PG\cU)\quad\text{on }(0,T),\qquad
\cU(0)=\u_0,
\end{split}
\end{equation}
where~$\PG$ is the global $L^2$-projection defined by~$\PG|_{I_m}=\PGm$, $m=1,\ldots, M$. 

\subsubsection{Existence and Uniqueness of Solutions}
For~$m\in\{1,\ldots,M\}$, integrating the strong form \eqref{eq:strong} of the cG method from~$t_{m-1}$ to~$t\in I_m$, we obtain the fixed point equation
\begin{equation}\label{eq:fp}
\cU(t)=\cU^{m-1}+\int_{t_{m-1}}^t\PGm\F(\PGm\cU)\,\dd\tau,
\end{equation}
with~$\cU^{m-1}$ from~\eqref{eq:Um1}. Under the assumption
\begin{equation}\label{eq:Lm}\tag{F2}
L_m:=\sup_{t\in I_m}\sup_{\genfrac{}{}{0pt}{}{v,w\in X}{v\neq w}}\frac{\Norm{\F(v)-\F(w)}_X}{\Norm{v-w}_X}<\infty,
\end{equation}
and using a contraction argument, we will prove that a unique solution of~\eqref{eq:fp} in~$\P^{r_m}(I_m;X)$ exists; we note, however, that existence and uniqueness of (local) solutions can be established under far more local conditions; cf., e.g., \cite{HolmWihler:15}.

\begin{proposition}
Suppose that the time steps in the cG scheme~\eqref{eq:cgfem} are selected such that~$k_m<\nicefrac{\sqrt2}{L_m}$, where~$L_m$ is the constant from~\eqref{eq:Lm}. Then, there exists a unique solution of~\eqref{eq:Um1}--\eqref{eq:cgfem}, for any~$m=1,\ldots,M$ (and for any initial value~$u_0\in X$). In particular, the given bound on the local time steps~$k_m$ is independent of the polynomial degree vector~$\bm r$.
\end{proposition}

\begin{proof}
For~$1\le m\le M$, defining the operator~$\mathcal{T}_m:\,\mathbb{P}^{r_m}(I_m;X)\to \mathbb{P}^{r_m}(I_m;X)$ by
\[
\mathcal{T}_m(V):=\cU^{m-1}+\int_{t_{m-1}}^t\PGm\F(\PGm V)\,\dd\tau,\qquad V\in\mathbb{P}^{r_m}(I_m;X),
\]
we have
\begin{align*}
\Norm{\mathcal{T}_m(V)-\mathcal{T}_m(W)}_X
&\le \left\|\int_{t_{m-1}}^t\PGm(\F(\PGm V)-\F(\PGm W))\,\dd\tau\right\|_X\\
&\le \int_{t_{m-1}}^{t}\Norm{\PGm(\F(\PGm V)-\F(\PGm W))}_X\,\dd\tau,
\end{align*}
for any~$V,W\in\mathbb{P}^{r_m}(I_m;X)$, and for all~$t\in I_m$. Employing the Cauchy-Schwarz inequality as well as the stability of~$\PGm$ (with constant~1), and making use of~\eqref{eq:Lm}, leads to
\begin{align*}
\Norm{\mathcal{T}_m(V)-\mathcal{T}_m(W)}_X^2
&\le (t-t_{m-1})\Norm{\PGm(\F(\PGm V)-\F(\PGm W))}_{L^2(I_m;X)}^2\\
&\le (t-t_{m-1})\Norm{\F(\PGm V)-\F(\PGm W)}_{L^2(I_m;X)}^2\\
&\le L_m^2(t-t_{m-1})\Norm{\PGm(V-W)}_{L^2(I_m;X)}^2\\
&\le L_m^2(t-t_{m-1})\Norm{V-W}_{L^2(I_m;X)}^2.
\end{align*}
Integrating this inequality over~$I_m$, we see that~$\mathcal{T}_m$ is a contraction provided that~$L_m^2k_m^2<2$. Thus, making use of Banach's fixed point theorem, we see that~$\mathcal{T}_m$ has a unique fixed point in~$\Pp$.
\end{proof}

\subsubsection{Nodal Norm Exactness}

The non-standard appearance of the $L^2$-projection operator~$\PGm$ in~\eqref{eq:cgfem} leads to the preservation of~\eqref{eq:0} at each time node~$t_m$, $m=0,1,2,\ldots, M$. 

\begin{proposition}\label{pr:0}
If~$\F$ satisfies the structural assumption~\eqref{eq:F1}, then any solution~$\cU$ of the cG scheme~\eqref{eq:cgfem} satisfies~$\{\cU(t_{m})\}_{m=0}^{M}\subset\S$, where~$\S$ is the sphere from~\eqref{eq:S}.
\end{proposition}

\begin{proof}
For~$m=1,\ldots,M$, multiplying the strong form~\eqref{eq:strong} by~$\PG\cU$, and integrating from~$t=0$ to~$t=t_m$, we recall~\eqref{eq:F1} to infer that
\begin{align*}
0
&=\int_{0}^{t_m}\ip{\cUp-\PG\F(\PG\cU),\PG\cU}\dd t\\
&=\int_{0}^{t_m}\ip{\cUp-\F(\PG\cU),\PG\cU}\dd t\\
&=\int_{0}^{t_m}\ip{\cUp,\PG\cU}\dd t.
\end{align*}
Since~$\cUp|_{I_i}\in\Pii$, $i=1,\ldots,m$, we arrive at
\begin{align*}
0
=\int_{0}^{t_m}\ip{\cUp,\cU}\dd t
=\frac12\int_{0}^{t_m}\frac{\dd}{\dt}\norm{\cU(t)}^2=\frac12\norm{\cU(t_m)}^2-\frac12\norm{\u_0}^2.
\end{align*}
Thus, we conclude~$\norm{\cU(t_m)}=\norm{\u_0}$, for any~$m=1,\ldots, M$.
\end{proof}


\section{\emph{A Posteriori} Error Analysis}\label{sc:apost}

We will now derive an \emph{a posteriori} error estimate for the cG method~\eqref{eq:Um1}--\eqref{eq:cgfem}, and provide a few numerical experiments. To this end, we suppose that, in addition to~\eqref{eq:F1}, the \emph{lower adjoint} property
\begin{equation}\label{eq:F2}\tag{F3}
\ip{\F(v),w}\ge -\ip{v,\F(w)}\qquad\forall v,w\in X
\end{equation}
holds true; again, as for~\eqref{eq:F1}, this includes the case of linear, skew-symmetric operators. Combining~\eqref{eq:F1} and~\eqref{eq:F2}, it follows that
\begin{equation}\label{eq:F3}
\ip{\F(v)-\F(w),v-w}=-\ip{\F(v),w}-\ip{v,\F(w)}\le 0,
\end{equation}
for any~$v,w\in X$.

\subsection{\emph{A Posteriori} Error Bound}
Following the approach proposed in~\cite{AkrivisMakridakisNochetto:09}, we define a reconstruction~$\Uh:\,[0,T]\to X$ of the cG solution~$\U$ from~\eqref{eq:Um1}--\eqref{eq:cgfem} based on raising each of the local $L^2$-projections appearing in~\eqref{eq:fp} by one degree, i.e., $\PGm\F(\PGm\U)\mapsto\PGp\F(\PGp\U)\equiv\PGp\F(\U)$. More precisely, for any~$m=1,\ldots,M$, we let
\begin{equation}\label{eq:rec}
\Uh|_{\overline I_m}(t):=\U^{m-1}+\int_{t_{m-1}}^t\Pi^{r_m}\F(\U)\,\dd\tau+\mathfrak{U}_{\cG}^{m-1}\int_{t_{m-1}}^t\L^m(\tau)\,\dd\tau,\qquad t\in\overline I_m.
\end{equation}
Here, noticing~\eqref{eq:strong} as well as the fact that
\[
\int_{I_m}\PGm\F(\U)\,\dd\tau=\int_{I_m}\PGp\F(\U)\,\dd\tau,
\]
we define
\begin{align*}
\mathfrak{U}_{\cG}^{m-1}:&=\int_{I_m}\PGm(\F(\PGm\U)-\F(\U))\,\dd t
=\int_{I_m}(\Up-\PGp\F(\U))\,\dd t\\
&=\U^{m}-\U^{m-1}-\int_{I_m}\PGp\F(\U)\,\dd t\in X
\end{align*}
in order to incorporate a discrete $\delta$-distribution $\L^m:\,\overline I_m\to\P^{r_m}(I_m;\mathbb{R})$ into~\eqref{eq:rec}; it is given in weak form by
\[
t\mapsto\L^m(t):\qquad\int_{I_m}\L^m(t)q(t)\dt=q(t_{m-1})\quad\forall q\in\P^{r_m}(I_m;\mathbb{R});
\]
see~\cite{SchotzauWihler:10} for details. As in the analysis of discontinuous Galerkin time stepping methods, the appearance of~$\L^m$ in~\eqref{eq:rec} ensures that~$\Uh(t_{m-1})=\U^{m-1}$ and~$\Uh(t_{m})=\U^m$, and, hence, that~$\Uh\in C^0([0,T];X)$. Then, we define the error~$\eh:=u-\Uh$, and assume that the exact solution~$u$ of~\eqref{eq:12} satisfies
\begin{align}\label{eq:reg}
\up&\in L^1((0,T);X);
\end{align}
this, in turn, implies that~$u\in C^0([0,T];X)$ (cf.~\cite[Proof of Lemma~7.1]{Roubicek:05}). Thus, we have
\begin{align*}
\frac12\frac{\dd}{\dt}\norm{\eh(t)}^2
=\ip{\up(t)-\Uhp(t),\eh},
\end{align*}
for any~$t\in I_m$. Therefore, applying~\eqref{eq:F3}, we obtain
\begin{align*}
\frac{\dd}{\dt}\norm{\eh(t)}^2
&=2\ip{\F(t,\u(t))-\Uhp(t),\eh(t)}\\
&=2\ip{\F(t,\u(t))-\F(t,\Uh(t)),\eh(t)}+2\ip{\F(t,\Uh(t))-\Uhp(t),\eh(t)}\\
&\le2\ip{\F(t,\Uh(t))-\Uhp(t),\eh(t)}.
\end{align*}
Furthermore, employing the Cauchy-Schwarz inequality, yields
\[
\frac{\dd}{\dt}\norm{\eh(t)}^2
\le 2\|\F(t,\Uh(t))-\Uhp(t)\|_X\|\eh(t)\|_X.
\]
Noticing that~$\eh(0)=0$, and integrating the above differential inequality, we infer that
\[
\norm{\eh(t)}\le \int_{0}^t\|\F(\Uh)-\Uhp\|_X\dd\tau,\qquad t\in[0,T].
\]
This leads to the bound
\begin{align*}
\|\eh\|_{L^\infty((0,t);X)}
\le \Norm{\F(\Uh)-\Uhp}_{L^1((0,t);X)},\qquad t\in[0,T].
\end{align*}
Hence, applying the triangle inequality, we have proved the ensuing \emph{a posteriori} error bound.

\begin{theorem}\label{thm:main}
Suppose that the operator~$\F$ satisfies~\eqref{eq:F1} and~\eqref{eq:F2}, and that the exact solution~$u$ of~\eqref{eq:12} fulfills~\eqref{eq:reg}. Then, if~$\cU\in\CG$ solves the cG scheme~\eqref{eq:Um1}--\eqref{eq:cgfem}, the \emph{a posteriori} error bound
\begin{equation}\label{eq:main}
\begin{split}
\|\u-\U\|_{L^\infty((0,t);X)}
&\le \Norm{\F(\Uh)-\Uhp}_{L^1((0,t);X)}
+\|\U-\Uh\|_{L^\infty((0,t);X)},
\end{split}
\end{equation}
holds true for any~$t\in[0,T]$. Here, $\Uh$ is the reconstruction from~\eqref{eq:rec}.
\end{theorem}

\begin{remark}
The above derivations would actually work without applying a reconstruction technique, i.e., for~$\Uh=\U$. Then, however, the resulting \emph{a posteriori} error estimate~\eqref{eq:main} would be sub-optimal; cf.~\cite{AkrivisMakridakisNochetto:09}.
\end{remark}


\subsection{Numerical Example}

In order to test the \emph{a posteriori} error estimate from Theorem~\ref{thm:main}, we consider the linear skew-symmetric $3\times 3$-system
\[
\dot{\bm u}(t)=
\begin{pmatrix}
0 & -2t\cos(3e^{-t}) & -2t\sin(3e^{-t})\\
2t\cos(3e^{-t}) & 0 & 3e^{-t}\\
2t\sin(3e^{-t}) & -3e^{-t} & 0
\end{pmatrix}\bm u(t),\qquad t>0,
\]
with the initial condition~$\bm u(0)=(1,0,0)^\top$. Its exact solution is smooth, and can be represented explicitly by
\[
\bm u:\,[0,\infty)\to\mathbb{R}^3,\qquad \bm u(t)=\begin{pmatrix}
\cos(t^2),
\sin(t^2)\cos(3e^{-t}),
\sin(t^2)\sin(3e^{-t})
\end{pmatrix}^\top.
\]

We perform two different sets of computational experiments: Firstly, for~$T=4$, we study the $L^\infty((0,T);\mathbb{R}^3)$-error of the cG method for fixed polynomial degrees~$r=r_m=\{1,2,3,4\}$ on all time steps~$m\ge 1$. Here, we employ a uniform time partition with time step lengths~$k=k_m\to 0$. In the context of the finite element method (FEM), this is referred to as an $h$-version approach. From Figure~\ref{fig:hFEM}, we clearly see that the cG method converges of order~$\mathcal{O}(k^{r+1})$ in the~$L^\infty$-norm, which is line with the results on the standard cG method from~\cite{EstepFrench:94}. In addition, we observe that the \emph{a posteriori} error estimator from Theorem~\ref{thm:main}, i.e., the right-hand side of~\eqref{eq:main} exhibits the same rate.

Secondly, we apply a spectral ($p$-version in the context of FEM) approach. Here, again for~$T=4$, we use a fixed time step size~$k=k_m=1$ on all elements, and consider the performance of the $L^\infty((0,T);\mathbb{R}^3)$-error as well as of the \emph{a posteriori} error estimator from~\eqref{eq:main} for increasing polynomial degrees~$r=r_m\in\{1,\ldots,16\}$ (again, the same polynomial degree on all time steps is employed). In Figure~\ref{fig:pFEM}, these quantities are plotted in a semi-log coordinate system. The nearly straight lines indicate that the method is able to achieve exponential rates of convergence as the polynomial degree is increased, and that the same behavior is observed for the error estimator.

\begin{figure}
\includegraphics[width=0.475\linewidth]{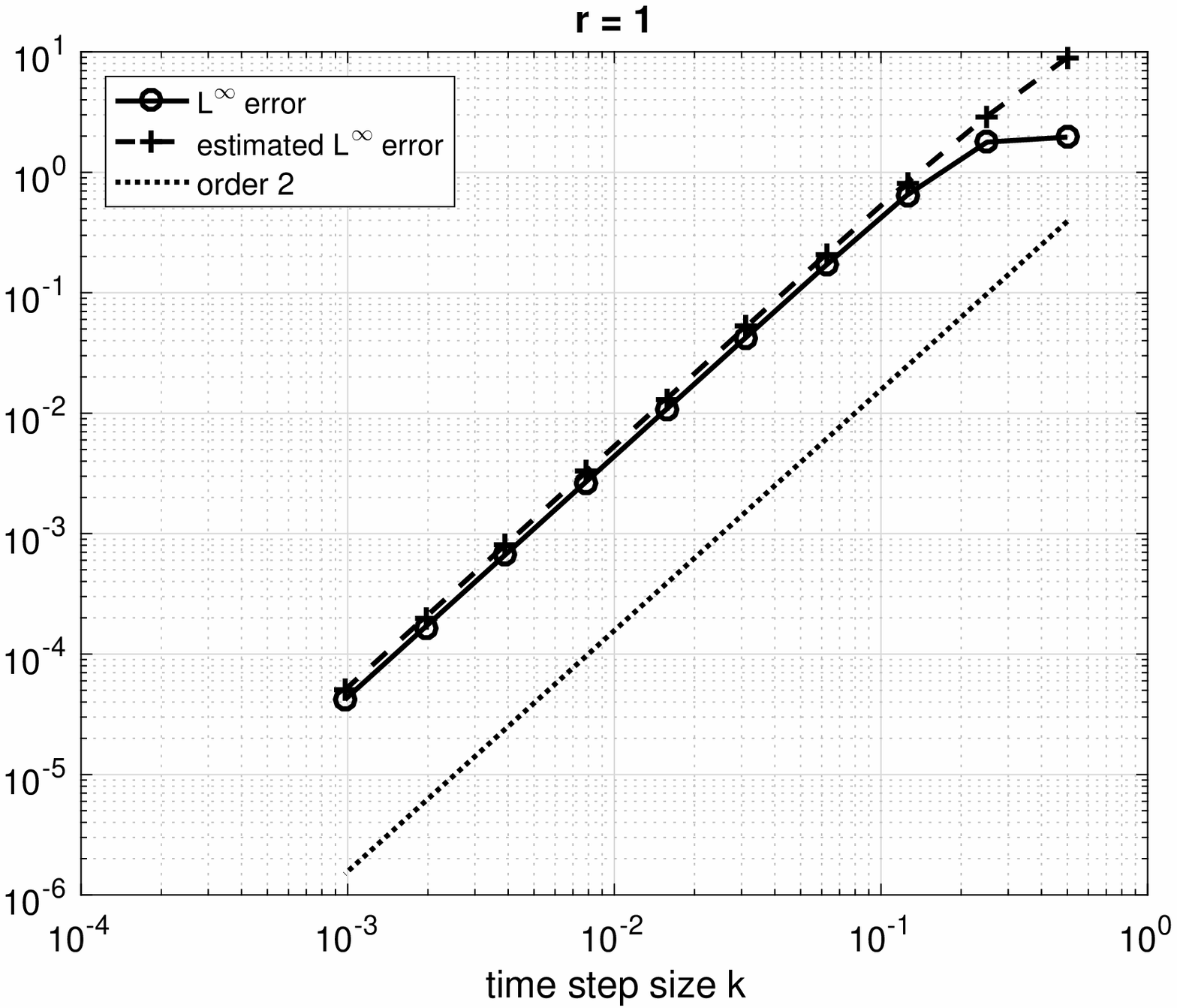}\hfill
\includegraphics[width=0.475\linewidth]{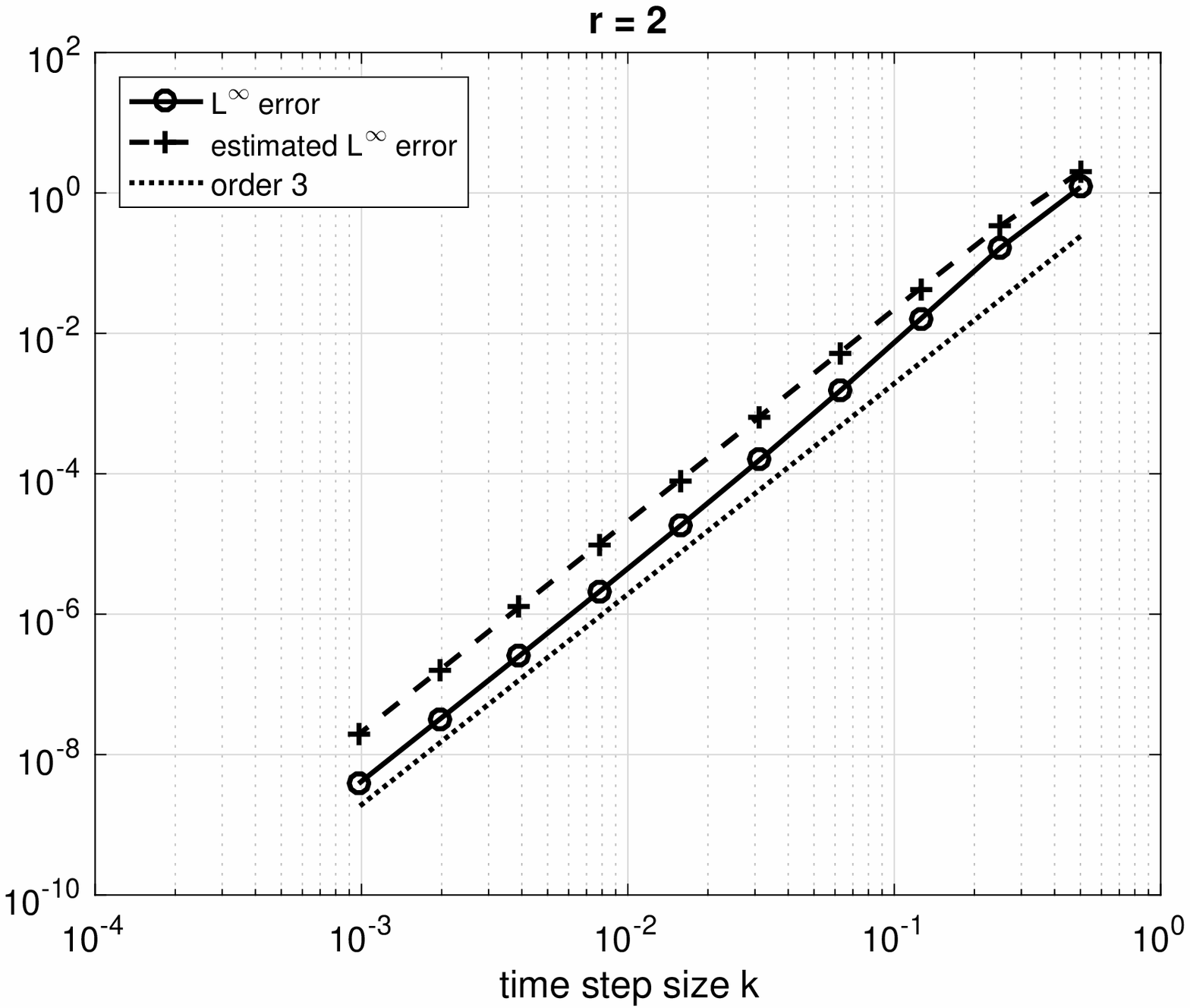}\\[2ex]
\includegraphics[width=0.475\linewidth]{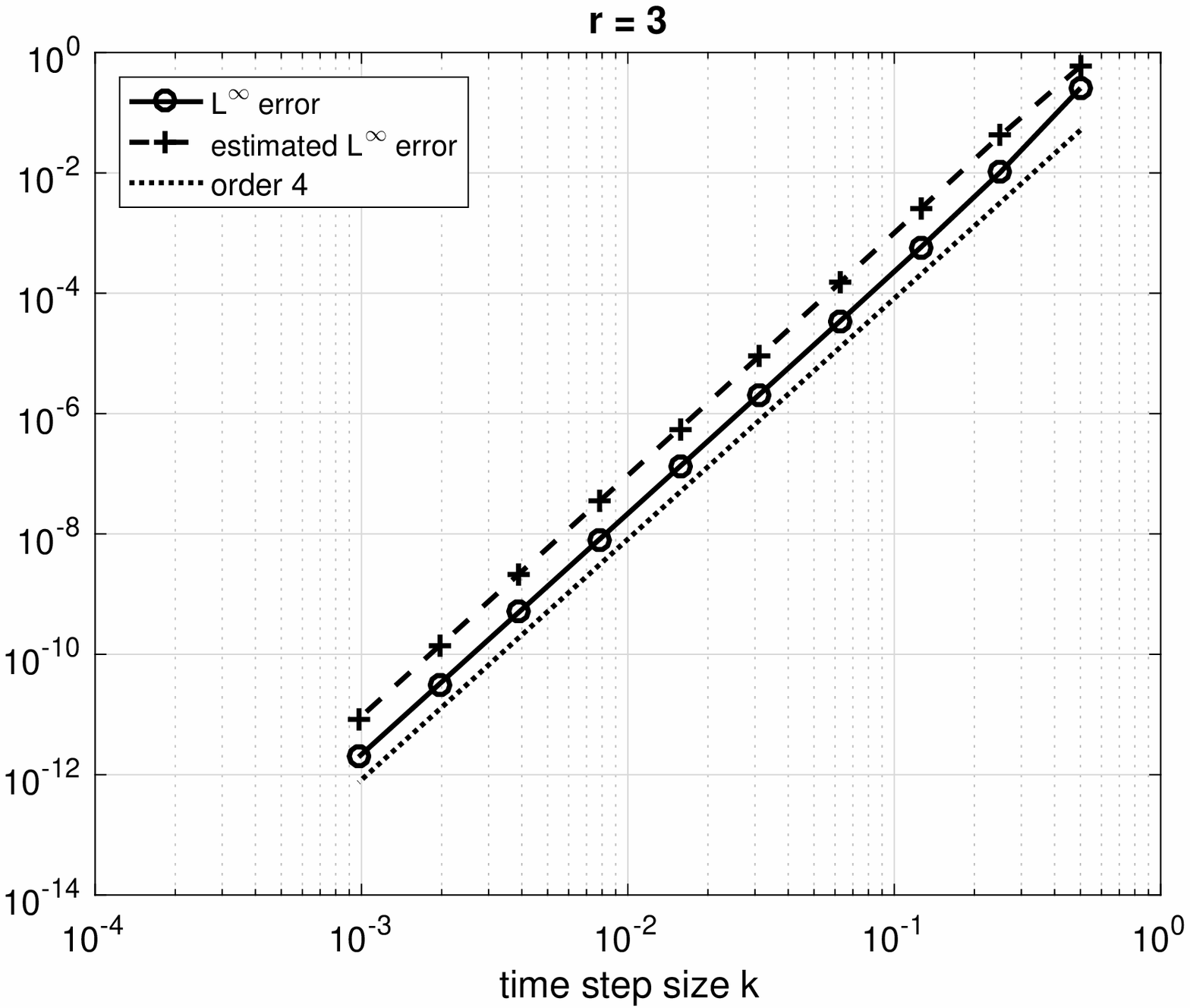}\hfill
\includegraphics[width=0.475\linewidth]{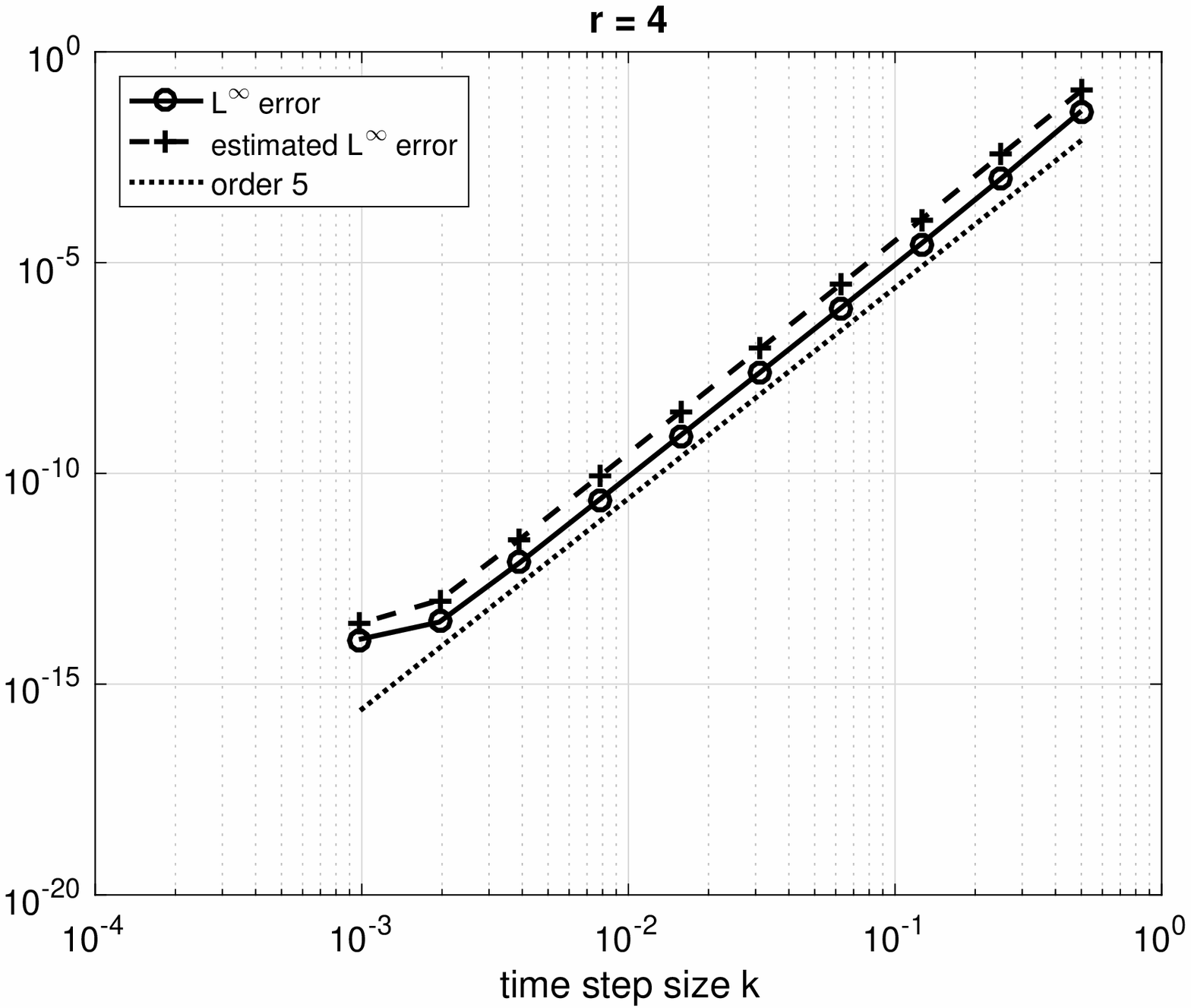}
\caption{Performance data for decreasing time steps with~$r\in\{1,2,3,4\}$.}
\label{fig:hFEM}
\end{figure}

\begin{figure}
\includegraphics[width=0.475\linewidth]{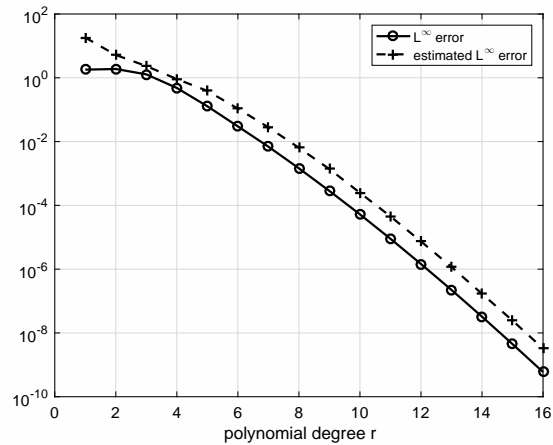}
\caption{Performance data for increasing polynomial degrees and fixed time step size~1.}
\label{fig:pFEM}
\end{figure}

\section{Conclusions}

In this note, we have presented a continuous Galerkin time stepping method for norm-preserving initial value problems. Moreover, we have derived an \emph{a posteriori} error estimate for the $L^\infty$-norm. For a smooth solution, our numerical experiments demonstrate that the estimator is of optimal algebraic order as the time steps tend to zero, and, in addition, exhibits exponential convergence for increasing polynomial degrees. Finally, we remark that the error estimator could be employed in the design of $h$- and $hp$-type adaptive time step procedures; cf., e.g., \cite{EstepFrench:94} and \cite{SchotzauWihler:10}, respectively.


\bibliographystyle{plain}
\bibliography{myrefs}
\end{document}